\documentclass[a4paper, 11pt, reqno]{amsart}
\usepackage[usenames,dvipsnames]{color}
\usepackage{amssymb, amsmath, latexsym, graphics, graphicx}
\newtheorem{theorem}{Theorem}[section]
\newtheorem{lemma}[theorem]{Lemma}
\newtheorem{corollary}[theorem]{Corollary}

\newtheorem{proposition}[theorem]{Proposition}

\theoremstyle{definition}
\newtheorem{remark}[theorem]{Remark}
\makeatother

\def\B{{\mathbb B}}

\def\S{{\mathcal S}}
\def\M{{\mathcal M}}
\def\R{{\mathcal R}}
\def\L{{\mathcal L}}
\def\D{{\mathcal D}}
\def\H{{\mathcal H}}
\def\J{{\mathcal J}}

\begin{document}
\title[Preserving Green's relations]
{Linear isomorphisms preserving Green's relations for matrices over semirings}

\maketitle

\begin{center}

ALEXANDER GUTERMAN\footnote{
Faculty of Algebra, Department of Mathematics and
Mechanics, Moscow State University, GSP-1, 119991 Moscow, Russia.},
MARIANNE JOHNSON\footnote{School of Mathematics, University of Manchester,
Manchester M13 9PL, UK. Email \texttt{Marianne.Johnson@maths.manchester.ac.uk}.} \\
and
MARK KAMBITES\footnote{School of Mathematics, University of Manchester,
Manchester M13 9PL, UK. Email \texttt{Mark.Kambites@manchester.ac.uk}.}

\date{\today}
\thanks{}
\end{center}

\begin{abstract} In this paper we characterize those linear bijective maps on the monoid of all $n\times n$ square matrices over an anti-negative semifield which preserve and strongly preserve each of Green's equivalence relations $\L$, $\R$, $\D$, $J$ and the corresponding three pre-orderings $\leq_\L$, $\leq_\R$, $\leq_J$. These results apply in particular to the tropical and boolean semirings, and for these two semirings we also obtain corresponding results for the $\H$ relation. 
\end{abstract}

\section{Introduction}
Given an equivalence or order relation on a certain algebraic system, it is natural to ask: what are the transformations that can be performed on this system that leave the relation invariant? In the case where the algebraic system under consideration is a matrix algebra over a field, the investigation of such transformations dates back to the following result of  Frobenius \cite{Fr}, which gives a characterisation of the bijective complex linear transformations for which the determinant is an invariant:
\begin{theorem} {\rm(Frobenius 1897)}
     \label{thm1.1} Let ${\mathbb C}$ be the field of complex numbers, and let $T:M_n({\mathbb C})\to M_n({\mathbb C})$ be a
bijective linear transformation such that $\det T(X)=\det X$
for all matrices $X\in M_n({\mathbb C})$. Then there exist invertible matrices $U,V\in
M_n({\mathbb C})$, with $\det (UV)=1$, such that either $T(X)=UX V$ for all matrices $ X\in
M_n({\mathbb C})$, or
$T(X)=U X^T V $ for all $X\in M_n({\mathbb C})$,  where $X^T$ denotes the transposed
matrix.
     \end{theorem}

This result was subsequently generalized by Schur~\cite{Shu}, who gave a characterisation of
maps preserving
 all subdeterminants of any fixed order~$r$. Later Dieudonn\'e~\cite{Di2}
proposed a new approach to classify such transformations, based on
the fundamental theorem of projective geometry. Dieudonn\'e obtained
a standard characterization of bijective linear 
maps preserving
the set of singular matrices over arbitrary fields.

Following on from these initial investigations, many authors have studied the problem of determining the linear operators on the $n \times n$ matrix algebra $M_n(F)$ over a field $F$ that
leave certain matrix relations, subsets, or properties invariant
(see the surveys~\cite{LT,P} for the details). In the last two decades much attention has been paid to the investigation of maps preserving
different
 invariants for matrices over various \emph{semirings}, where completely different
techniques are necessary to obtain a classification of linear transformations
with certain preserving properties; see~\cite[Section
9.1]{P} and references therein for more details.

In this paper we consider Green's relations for the full monoids of $n \times n$ matrices over anti-negative semifields, which include most notably the tropical and boolean semirings. Green's relation are five equivalence relations ($\L$, $\R$, $\H$, $\D$ and $\J$) and three pre-orders ($\le_\R$, $\le_\L$ and $\le_J$) which can be defined on any semigroup. They encapsulate the divisibility relations between elements and hence together describe the structure of the maximal subgroups and principal left, right, and two-sided ideals of the semigroup.  Green's relations play an important
role in many aspects of semigroup theory.  A natural problem is to characterize all bijective linear maps which preserve each of the orders or equivalence  relations defined by Green. In Section 2 we briefly recall the necessary definitions,
pointing out
that in the case of the tropical semiring (or more generally any anti-negative semiring without zero divisors \cite[Theorem 2.14]{BG}) the bijective linear maps have a very restricted form. In Section 3 we give a complete characterisation of those bijections which preserve each of the relations $\L$, $\R$, $\D$, $\J$ and the pre-orders $\le_\R$, $\le_\L$ and $\le_\J$. In Section 4 we characterise the bijections which preserve the relation $\H$ in $M_n(\S)$ for some restricted classes of anti-negative semifields, including the Boolean and tropical semifields.

\section{Green's relations and factor rank in matrix semigroups}
Let $\M$ be a monoid. For $a,b\in\M$ we say that:
\begin{itemize}
\item[(i)] $a\le_\R b$ if and only if $a\M\subseteq b\M$, that is, if there exists $s\in \M$ with $a=bs$. We say that $a\R b$ if $a\le_\R b$ and $b\le_\R a$, or in other words if $a$ and $b$ generate the same principal right ideal of $\M$.
\item[(ii)] $a\le_\L b$ if and only if $\M a\subseteq \M b$, that is, if there exists $s\in \M$ with $a=sb$. We say that $a\L b$ if $a\le_\L b$ and $b\le_\L a$, or in other words if $a$ and $b$ generate the same principal left ideal of $\M$.
\item[(iii)] $a\le_\J b$ if and only if $\M a\M\subseteq \M b\M$, that is, if there exist $s,t\in \M$ with $a=sbt$. We say that $a\J b$ if $a\le_\J b$ and $b\le_\J a$, or in other words, if $a$ and $b$ generate the same principal two-sided ideal of $\M$.
\item[(iv)] $a\D b$ if and only if there exists $ c \in \M$ such that $a\R c$ and $c\L b$ (or equivalently, if there there exists $c \in \M$ such that $a\L c$ and $c\R b$ ).
\item[(v)] $a\H b$ if and only if $a\R b$ and $a\L b$.
 \end{itemize}

We recall that a \emph{semiring} is a set ${\mathcal S}$ together with two binary operations, addition and multiplication, such that $\S$ is a commutative monoid under addition (with identity denoted by $0_{\S}$); $\S$ is a semigroup under multiplication (with identity, if any, denoted by $1_{\S}$); multiplication is distributive over addition on both sides; and $s0_{\S}=0_{\S}s=0_{\S}$ for all $s\in \S$. The semiring $\S$ is said to be \textit{commutative} if the multiplication is commutative. It is said to be \emph{anti-negative} if $a + b = 0_{\S}$ implies that both $a$ and $b$ are equal to $0_{\S}$. It is a \textit{semifield} if $\S \setminus \lbrace 0_{\S} \rbrace$ is an abelian group under multiplication.

\begin{remark}\label{rem_2element}
If $S  = \lbrace 0_{\S}, 1_{\S} \rbrace$ is a 2-element semifield then it is easy to see that the axioms determine the values of all sums and products except for $1_{\S} + 1_{\S}$.
Setting $1_{\S} + 1_{\S} = 0_{\S}$ yields the field $\mathbb{Z}_2$, while setting $1_{\S} + 1_{\S} = 1_{\S}$ yields the \textit{boolean semiring}, which we shall denote by $\B$. Since $\mathbb{Z}_2$ is not anti-negative and the definition of a semifield implies $0_{\S} \neq 1_{\S}$, this means that $\B$  is the unique anti-negative semifield with strictly fewer than three elements. Some of our results about anti-negative semifields will be established by one argument requiring the ability to take three distinct elements, and a separate argument for the special case of $\B$.
\end{remark}

\begin{proposition}\label{prop_notallroots}
Every anti-negative semifield $\S$ except for $\mathbb{B}$ contains an invertible element $k$ such that $k^2 \neq 1_S$.
\end{proposition}
\begin{proof}
Assume for contradiction that $\S \neq \mathbb{B}$ is an anti-negative semifield in which all invertible elements square to give $1_{\S}$. Since
$\mathbb{B}$ is the unique anti-negative semifield with fewer than three elements (by Remark~\ref{rem_2element} above), we may choose distinct non-zero
elements $x, y \in \S$. Since $\S$ is anti-negative, $x+y$ is also non-zero. Commutativity of $\S$ together with our assumption yields:
$$1_{\S} = (x+y)(x+y) = x^2+xy+xy+y^2 = 1_{\S}+xy+xy+1_{\S}.$$
Multiplying both sides by $x$ and using $x^2 = 1_{\S}$ then gives $x = x+y+y+x$, while a dual argument gives $y = y+x+x+y$, contradicting the assumption that $x$ and $y$ are distinct.
\end{proof}

\begin{remark}\label{idmpt_tropical}
A semiring $\S$ is called \emph{idempotent} if $a + a = a$ for all $a \in \S$. An idempotent semiring is necessarily anti-negative; indeed if $a+b=0_{\S}$ then $0_{\S} = a+b = a+ a + b = a + 0_{\S} = a$ and a dual argument gives $0_{\S} = b$. For example, the \emph{tropical semifield} $\mathbb{R}_{{\rm max}}:=\mathbb{R} \cup \{-\infty\}$ together with addition given by taking the maximum and multiplication given by extending \textit{addition} of real numbers so as to make $-\infty$ a zero element, is an idempotent and hence anti-negative semifield.
\end{remark}

Let $\S$ be a semiring.  We write $\S^{i \times j}$ for the set of $i \times j$ matrices over $\S$, which forms an \emph{$\S$-module} in the obvious way. We write $M_n(\S)$ for $\S^{n \times n}$ viewed as a semigroup under the matrix multiplication induced by the operations in $\S$. If $\S$ contains a multiplicative identity element $1_{\S}$, then $M_n(\S)$ is a monoid, with obvious identity element.

Given an anti-negative semifield, our aim is to characterise those bijective $\S$-linear maps
$$T: M_n(\S) \rightarrow M_n(\S),$$
which preserve each of the pre-orders $\le_\R, \le_\L, \le_\J$ and each of the relations $\R, \L, \J, \D, \H$. We say that the map $T$ \textit{preserves} the relation ${\mathcal P}$ if $A{\mathcal P} B$ implies that  $T(A){\mathcal P} T(B)$ for all $A,B\in M_n(\S)$. We say that map  $T$ \textit{strongly preserves} ${\mathcal P}$ if for all $A,B\in M_n(\S)$ we have $A{\mathcal P} B$ if and only if $T(A){\mathcal P} T(B)$. For example, it is straightforward to verify that any transformation of the form $T(A)=PAQ$, where $P$ and $Q$ are invertible matrices over $\S$ will strongly preserve each of the described relations. The transposition transformation strongly preserves  $\le_\J$, $\J$, $\D$ and $\H$ but it does not preserve any of the other relations (the $\L$ and $\R$ relations being transposed).

We begin by recalling a number of module-theoretic characterisations of Green's relations for matrix semigroups. For $A \in M_n(\S)$ we write ${\rm Row}_{\S}(A)$ to denote the left $\S$-submodule of $\S^{1\times n}$ generated by the rows of $A$ and ${\rm Col}_{\S}(A)$ to denote the right $\S$-submodule of $S^{n\times 1}$ generated by the columns of $A$. The following statement is well-known and can be found in the literature
for example as \cite[Proposition 4.1]{HK}.

\begin{theorem} \label{thm:HK}
Let $\S$ be a semiring with multiplicative identity element,  and $n \in \mathbb{N}$. For $A, B \in M_n(\S)$ we have
\begin{itemize}
\item[(i)] $A \leq_\L B$ if and only if ${\rm Row}_{\S}(A) \subseteq {\rm Row}_{\S}(B)$;
\item[(ii)] $A \leq_\R B$ if and only if ${\rm Col}_{\S}(A) \subseteq {\rm Col}_{\S}(B)$;
\item[(iii)] $A \L B$ if and only if ${\rm Row}_{\S}(A) = {\rm Row}_{\S}(B)$;
\item[(iv)] $A \R B$ if and only if ${\rm Col}_{\S}(A) = {\rm Col}_{\S}(B)$;
\item[(v)] $A \H B$ if and only if ${\rm Row}_{\S}(A) = {\rm Row}_{\S}(B)$ and ${\rm Col}_{\S}(A) = {\rm Col}_{\S}(B)$.
\end{itemize}
\end{theorem}

For semirings (or even semifields) in general, the $\D$ and $\J$ relations are harder to characterise in terms of row and column spaces. However, characterisations are known for many important semirings including in particular the tropical semiring $\mathbb{R}_{\rm max}$ (see
\cite{ABG,HK,JK12,WJK}).

Recall that the {\em factor rank\/} $f(A)$ of a matrix $A \in \S^{n \times m}$ is the smallest positive integer $k$ such that $A=BC$ for some $B \in \S^{n \times k}$ and $C \in \S^{k \times m}$. By convention, a matrix with only zero entries has factor rank $0$. Factor rank is also
known as \textit{Schein rank} (especially over the Boolean semiring) or \textit{Barvinok rank} (especially over the tropical semiring).

\begin{remark}\label{rem_rank2} Let $\S$ be a semifield and suppose that $M \ \in \ M_n(\S)$ is a matrix with exactly four non-zero entries which are arranged in a square submatrix as follows:
$$\left(\begin{array}{c c} a & b \\c& d\end{array}\right).$$
Then it is easy to see that $M$ has factor rank at most $2$ and, moreover, $M$ has factor rank $2$ if and only if $ad \neq bc$.
\end{remark}

\begin{remark}\label{rem_factorgreen}
Factor rank respects the $\J$-order, in the sense that $A \leq_\J B$ implies $f(A) \leq f(B)$ (see for example \cite[Corollary~8.3]{JK12}). It follows
that it is an invariant of $\H-$, $\L$-, $\R-$, $\D-$ and $\J$-classes. Note that in general $A\le_\R B$ and $f(A)=f(B)$ do not suffice to conclude that $A\R B$. For example, let $\S$ be the semiring of non-negative integers, and let $A$ and $B$ be the $n\times n$ matrices with $A_{1,1}=2$, $B_{1,1}=1$ and all other entries equal to 0; then $A\le_\R B$ and $f(A)=f(B)$ but $A\not \!\! \R B$ since 2 is not invertible in $\S$.
\end{remark}

\section{Bijective linear transformations preserving Green's relations}
Let $E_{i,j}$ denote the matrix with $1_{\S}$ in the $(i,j)$th position and $0_{\S}$ elsewhere. It is clear that $M_n(\S)$ is a free $\S$-module of rank $n^2$ having the elements $E_{i,j}$ as basis. Thus any $\S$-linear transformation $T: M_n(\S) \rightarrow M_n(\S)$ is completely determined by the images $T(E_{i,j})$.

From now on we assume that $\S$ is an anti-negative semifield (such as for example the tropical or boolean semifield). In this case, it was shown  
in
\cite{BG} that the bijective $\S$-linear maps $T$ have a very restricted form. We write $[n]$ for the set $\lbrace 1, \dots, n \rbrace$.

\begin{lemma} \cite[Theorem 2.14]{BG}.
\label{lem:BG}
Let $\S$ be an anti-negative semifield and let $T:M_n(\S)\to M_n(\S)$ be an $\S$-linear transformation. The following are equivalent
\begin{itemize}
\item[(i)] $T$ is bijective;
\item[(ii)] $T$ is surjective;
\item[(iii)] There exists a permutation $\sigma \in {\rm Sym}([n]\times[n])$ and non-zero elements $\alpha_{i,j} \in \S$ such that $T(E_{i,j}) = \alpha_{i,j}E_{\sigma(i,j)}$ for all $i$ and $j$.
\end{itemize}
\end{lemma}

The following fact  belongs to folklore; its proof for distributive lattices and some other algebraic structures can be  found in \cite{Skor} and \cite{Il'in}. It is formulated without a proof  in \cite[Lemma 1]{Psh} for arbitrary anti-negative semirings. For completeness, we provide here a reduction Lemma~\ref{lem:BG} in the case of anti-negative semifields.

\begin{corollary}
 Let $\S$ be an anti-negative semifield.  The invertible elements of the monoid $M_n(\S)$ are precisely the monomial matrices, that is, those matrices containing exactly one non-zero element in each row and in each column.
\end{corollary}

\begin{proof}
Since every non-zero element of $\S$ is invertible it is clear that the monomial matrices are invertible; each can be written as a product of an invertible diagonal matrix and a permutation matrix. On the other hand, left multiplication by an invertible matrix $A$ yields a bijective $\S$-linear transformation on $M_n(\S)$, and so by Lemma \ref{lem:BG} there exist non-zero elements $\alpha_{i,j} \in \S$ and a permutation $\sigma \in {\rm Sym}([n] \times [n])$ such that $A \cdot E_{i,j} = \alpha_{i,j} E_{\sigma(i,j)}$ for all $i,j \in [n]$. Since $A = A \cdot (E_{1,1} + \cdots + E_{n,n}) = \alpha_{1,1} E_{\sigma(1,1)} + \cdots + \alpha_{n,n}E_{\sigma(n,n)}$, we deduce that $A$ must contain precisely $n$ non-zero entries. It now follows from the fact that $A$ is invertible (and hence cannot contain a zero row) that there must be exactly one non-zero entry in each row. An entirely similar argument, considering right multiplication by $A$, shows that $A$ must be column monomial.
\end{proof}

\begin{lemma} \label{lem:LR}
Let $\S$ be an anti-negative semifield and let $T:M_n(\S)\to M_n(\S)$ be a bijective $\S$-linear transformation. If $T$  preserves any one of $\L$, $\R$, $\le_\L$ or $\le_\R$, then there exist $\check{\sigma}, \hat{\sigma} \in {\rm Sym}([n])$ and non-zero elements $x_i, y_j \in S$ such that
$$T(E_{i,j}) = x_i y_j E_{\check{\sigma}(i),\hat{\sigma}(j)} \mbox { for all } i, j \in [n].$$
\end{lemma}
\begin{proof}
We prove the result for  $\L$; the claim for $\le_\L$ follows since a map preserving $\le_\L$ must clearly preserve $\L$, while the statements for $\R$ and $\le_\R$ are dual.

By Lemma~\ref{lem:BG}  there exist elements $\alpha_{i,j} \neq 0_{\S} \in \S$ for $i,j \in [n]$ and $\sigma \in {\rm Sym}([n]\times [n])$ such that $T(E_{i,j}) = \alpha_{i,j}E_{\sigma(i,j)}$ for all $i,j \in [n]$. It is an easy consequence of Theorem~\ref{thm:HK} that for any two non-zero elements $a,b \in \S$ and any $i,j,k,l \in [n]$ we have $(aE_{i,j})\L (bE_{k,l})$ if and only if $j=l$. Thus if $T$ preserves $\L$ then there exists $\hat{\sigma} \in {\rm Sym}([n])$ such that
$$\{\sigma(1,j), \ldots, \sigma(n,j)\} = \{(1, \hat{\sigma}(j)), \ldots, (n, \hat{\sigma}(j))\}$$
for each $j$. (Informally, this means that $T$ maps `columns to columns', in the sense that the images of any two matrix units $E_{i,j}$ and $E_{k,j}$ must have their non-zero entries in columns with the same index.)

We claim that there exists $\check{\sigma} \in {\rm Sym}([n])$ such that $\sigma(i,j) = (\check{\sigma}(i), \hat{\sigma}(j))$ for all $i$ and $j$ (so that $T$ also maps `rows to rows'). Suppose not. Then there exist $i,j,j' \in[n]$ with $j \neq j'$ such that $\sigma(i,j) = (m ,\hat{\sigma}(j))$ and $\sigma(i,j') = (m', \hat{\sigma}(j'))$ for some $m \neq m'$. (In other words, the images of $E_{i,j}$ and $E_{i,j'}$ have their non-zero entries in different rows.)
To derive the desired contradiction, we now consider separately the cases where $|\S| \geq 3$ and where $\S = \B$ is the boolean semiring (see Remark~\ref{rem_2element} above).

Suppose first that $|\S| \geq 3$. By Proposition \ref{prop_notallroots} we may choose $\lambda \in \S$ such that $\lambda \ne 0_{\S}, 1_{\S}$ and $\lambda^2 \neq 1_{\S}$. Consider
\begin{eqnarray*}
A&=&T^{-1}(E_{m,\hat{\sigma}(j)}+ E_{m',\hat{\sigma}(j)} + E_{m,\hat{\sigma}(j')} + E_{m',\hat{\sigma}(j')}),\\
B&=&T^{-1}(\lambda E_{m,\hat{\sigma}(j)}+ E_{m',\hat{\sigma}(j)} + E_{m,\hat{\sigma}(j')} + \lambda E_{m',\hat{\sigma}(j')}).
\end{eqnarray*}
Since $m \neq m'$, $\hat{\sigma}(j) \neq \hat{\sigma}(j')$ and $\lambda^2 \neq 1_{\S}$, it is clear that $f(T(A))= 1$ and $f(T(B))=2$. Since factor rank is an $\mathcal{L}$-class invariant we see that $T(A)\not \!\! \L T(B)$. On the other hand, it follows easily from the definitions above that there exist non-zero $a,b,c,d \in \S$ and $k,l \neq i$ such that
$$A=aE_{i,j}+ bE_{k,j} + cE_{l,j'} + dE_{i,j'}, \quad B=\lambda aE_{i,j}+ bE_{k,j} + cE_{l,j'} + \lambda dE_{i,j'}.$$
Now using the fact that $\lambda$ is invertible in $\S$, we see that $A$ and $B$ have the same row space, and so by Theorem~\ref{thm:HK} we have $A \L B$, providing the required contradiction.

Now consider the case where $\S = \B$. Let $A = E_{i,j} + E_{i,j'}$ and $B$ be the matrix with $1_S$ in every entry in the $j$th and $j'$th columns and $0_S$ elsewhere. Then $A$ and $B$ have the same row space and so by Theorem~\ref{thm:HK} are $\L$-related. However, it is easy to see that the row space of $T(A)$ does not coincide with the row space of $T(B)$, which by Theorem~\ref{thm:HK} means that $T(A) \not\!\!\!{\L} T(B)$, again giving a contradiction.

It remains to show that we can find elements $x_i, y_j \in \S$ for $i,j \in [n]$ with $\alpha_{i,j} = x_i y_j$. (Notice that since each $\alpha_{i,j}$ is non-zero this will also imply that the $x_i$'s and $y_j$'s are non-zero.) If we let $R$ be the matrix with $R_{i,j} = \alpha_{i,j}$, what we wish to show is that $R$ has factor rank $1$. Note that for arbitrary permutation matrices $P, Q \in M_n(\S)$ it is easy to see that $R$ has factor rank $1$ precisely if $PRQ$ has factor rank $1$.

Let $C$ be the matrix whose entries are all $1_S$ and $D$ the matrix whose first row entries are all $1_S$ and other entries all $0_S$. Then $C$ and $D$ clearly have the same row space and so by Theorem~\ref{thm:HK} are $\L$-related. Since $T$ preserves $\L$, this means $T(C) \L T(D)$, and hence by Theorem~\ref{thm:HK} again, $T(C)$ and $T(D)$ have the same row space. However, from the defining properties of the $\alpha_{i,j}$ it is easy to see that $T(D)$ has only one non-zero row, which is equal to the $\check{\sigma}(1)$th row of $T(C)$. Moreover, $T(C) = PRQ$ for permutation matrices $P, Q \in M_n(\S)$ corresponding to the permutations $\check{\sigma}$ and $\hat{\sigma}$ respectively. Thus, every row of $T(C)$ must be a multiple of the $\check{\sigma}(1)$th row of $T(C)$, showing that $T(C)$, and hence $R$, has factor rank $1$, as required.
\end{proof}

\begin{lemma} \label{lem:D}
Let $\S$ be an anti-negative semifield and let $T:M_n(\S)\to M_n(\S)$ be a bijective $\S$-linear transformation. If $T$  preserves any of the relations $\D$, $\J$ or $\leq_\J$ then there exist $\check{\sigma}, \hat{\sigma} \in {\rm Sym}([n])$ and non-zero elements $x_i, y_j \in \S$ for $i, j \in [n]$ such that either
$$T(E_{i,j}) = x_i y_j E_{\check{\sigma}(i),\hat{\sigma}(j)} \mbox{ for all } i,j \in [n]; \mbox{  or }$$
$$T(E_{i,j}) = x_i y_j E_{\hat{\sigma}(j),\check{\sigma}(i)} \mbox{ for all } i,j \in [n].$$
\end{lemma}
\begin{proof}
Suppose $T$ preserves one of the given relations.
By Lemma~\ref{lem:BG}  there exist non-zero $\alpha_{i,j} \in \S$ for $i,j \in [n]$ and $\sigma \in {\rm Sym}([n]\times [n])$ such that $T(E_{i,j}) = \alpha_{i,j}E_{\sigma(i,j)}$ for all $i,j \in [n]$.

We claim that for any $i,j,j' \in [n]$, the matrices $T(E_{i,j})$ and $T(E_{i,j'})$ have their non-zero entries in either the same row or the same column. Indeed,
consider the matrices $E_{i,j}$ and $E_{i,j}+E_{i,j'}$. These clearly have the same column space, so by Theorem~\ref{thm:HK} we have
$(E_{i,j} + E_{i,j'} )\R E_{i,j}$ and hence also $(E_{i,j} + E_{i,j'}) \D E_{i,j}$,
$(E_{i,j} + E_{i,j'}) \J E_{i,j}$ and
$(E_{i,j} + E_{i,j'}) \leq_\J E_{i,j}$. Since $T$ preserves one of the latter three relations, we have one of
$T(E_{i,j} + E_{i,j'}) \D T(E_{i,j})$,
$T(E_{i,j} + E_{i,j'}) \J T(E_{i,j})$ and
$T(E_{i,j} + E_{i,j'}) \leq_\J T(E_{i,j})$, so we must have
$T(E_{i,j} + E_{i,j'}) \leq_\J T(E_{i,j})$ since this is the weakest of the three possibilities. Since factor rank respects the $\mathcal{J}$-order (Remark~\ref{rem_factorgreen}) we note that the factor rank of $T(E_{i,j} + E_{i,j'})$ cannot exceed the factor rank of $T(E_{i,j})$, which is clearly seen to be $1$. The
only way this can happen is if the two non-zero entries of
$T(E_{i,j} + E_{i,j'}) = T(E_{i,j}) + T(E_{i,j'})$ lie in either the same row or the same column.

A similar argument (using $\L$ instead of $\R$) establishes a dual claim that for any $i,i',j \in [n]$, $T(E_{i,j})$ and $T(E_{i'j})$ have their non-zero entries in either the same row or the same column.

Consider the sets of matrices of the form $\lbrace E_{i,k} \mid k \in [n] \rbrace$ (which we
will call the \textit{$i$th row set}) and of the form $\lbrace E_{k,j} \mid k \in [n] \rbrace$ (which we will call the $j$th \textit{column set}). A simple inductive argument using the two dual claims above shows that $T$ must map every column or row set to either a column or a row set. In fact, we claim that $T$ must either
\begin{itemize}
\item map row sets to row sets and column sets to column sets (which we call the \textit{standard case}); or
\item map row sets to column sets and columns sets to row sets (which we call the \textit{transpose case}).
\end{itemize}
Indeed, if it did neither of these things, then it would have to map one row set to a row set and another row set to a column set; since the different row sets are disjoint but every row set intersects every column set, this would contradict the fact that $T$ is a bijection.

In the standard case, we define $\check\sigma$  and $\hat\sigma$ to be such that $T$ maps the
$i$th row set to the $\check\sigma(i)$th row set and the $j$th column set to the $\hat\sigma(j)$th column set. Since $T$ is a bijection,
they are permutations of $n$, and it is easy to see that $\sigma(i,j) = (\check\sigma(i),  \hat\sigma(j))$. In the transpose case we define them so that
$T$ maps the $i$th row set to the $\check\sigma(i)$th column set and the
$j$th column set to the $\hat\sigma(j)$th row set; again they are permutations and this time
$\sigma(i,j) = (\hat\sigma(j), \check\sigma(i))$.

It remains to show that we can find elements $x_i$ and $y_i$ with $\alpha_{i,j} = x_i y_j$.  In the standard case, we define $R$, $C$ and $D$ exactly as in the proof of Lemma~\ref{lem:LR}. The same argument as used there shows that $C \L D$ and hence $C \D D$, $C \J D$ and $C \leq_\J D$.
By assumption $T$ preserves at least one of the latter three relations, so we have at least one of $T(C) \D T(D)$, $T(C) \J T(D)$ and $T(C) \leq_\J T(D)$, which means we must have $T(C) \leq_\J T(D)$ since this is the weakest of the three possibilities. Now $T(D)$ has exactly one non-zero row, so must have factor rank $1$, while $T(C) = PRQ$, for some permutation matrices $P, Q \in M_n(\S)$. Since factor rank respects the $\mathcal{J}$-order
(Remark~\ref{rem_factorgreen}) it follows that $T(C)$, and hence $R$, has factor rank at most one $1$, which is exactly what we required.

The transpose case is treated by a very similar argument. In particular, $D$ is taken this time to be the matrix with $1_{\S}$ in the first column, we deduce that $C \R D$, that $T(C) \leq_\J T(D)$ and that $T(D)$ has exactly one non-zero row, and hence again that $T(C) = (PRQ)^T$ has factor rank at most $1$.
\end{proof}

\begin{theorem} \label{t1}
Let $\S$ be an anti-negative semifield and $T:M_n(\S)\to M_n(\S)$ a bijective $\S$-linear transformation. Then the following are equivalent:
\begin{itemize}
\item[(i)] $T$ preserves $\R$;
\item[(ii)] $T$ preserves $\L$;
\item[(iii)] $T$ preserves $\le_\R$;
\item[(iv)] $T$ preserves $\le_\L$;
\item[(v)] there exist invertible (monomial) matrices $P,Q\in M_n(\S)$ such that for all $X\in M_n(\S)$, we have $T(X)=PXQ$.
\end{itemize}
\end{theorem}
\begin{proof}
It is straightforward to verify (directly or using Theorem~\ref{thm:HK}) that (v) implies the other four conditions.

Conversely, suppose that one of (i), (ii), (iii) or (iv) holds. Let $\check{\sigma}$, $\hat{\sigma}$, $x_i$ and $y_i$ be as given by Lemma~\ref{lem:LR}.
Let $P$ and $Q$ be the monomial
matrices given by $P_{\check{\sigma}(i),i} = x_i$ and $Q_{j,\hat{\sigma}(j)} = y_j$. The map $M_n(\S) \to M_n(\S), X \mapsto PXQ$ is clearly $\S$-linear, so it
suffices to show that it agrees with $T$ on the basis elements $E_{i,j}$. For each such, from the definitions of $P$ and $Q$,
$$(P E_{i,j} Q)_{rs} = \begin{cases}
x_i y_j & \textrm{ if } r = \check{\sigma}(i) \textrm{ and } s = \hat{\sigma}(j) \\
0 & \textrm{ otherwise}.
\end{cases}$$
so that $P E_{i,j} Q = x_i y_j E_{\check{\sigma}(i), \hat{\sigma}(j)} = T(E_{i,j})$ by Lemma~\ref{lem:LR}.
\end{proof}

\begin{theorem}\label{t2}
Let $\S$ be an anti-negative semifield and $T:M_n(\S)\to M_n(\S)$ a bijective $\S$-linear transformation. Then the following are equivalent:
\begin{itemize}
\item[(i)] $T$ preserves $\D$;
\item[(ii)] $T$ preserves $\J$;
\item[(iii)] $T$ preserves $\le_\J$;
\item[(iv)] there exist invertible (monomial) matrices $P,Q\in M_n(\S)$ such that either $T(X)=PXQ$ for all $X \in M_n(\S)$, or $T(X)=P X^T Q$ for all $X \in M_n(\S)$.
\end{itemize}
\end{theorem}
\begin{proof}
Again, it is straightforward to verify (directly or using Theorem~\ref{thm:HK}) that (iv) implies the other three conditions. For the converse, suppose
one of (i), (ii) or (iii) holds and let $\check{\sigma}$, $\hat{\sigma}$, $x_i$ and $y_i$ be given this time by Lemma~\ref{lem:D}.
In the standard case (in the terminology of the proof of Lemma~\ref{lem:D}) define $P$ and $Q$ as in the proof of Theorem~\ref{t1} and we have $T(X) = PXQ$ for all $X$ by exactly the same argument but using Lemma~\ref{lem:D} in place of Lemma~\ref{lem:LR} at the end.

In the transpose case, define $P$ and $Q$ to be the monomial matrices with
$P_{\hat{\sigma}(j),j} = y_j$ and $Q_{i,\check{\sigma}(i)} = x_i$. The map $X \mapsto PX^TQ$ is easily seen to be $\S$-linear, and a simple calculation gives
$$(P (E_{i,j})^T Q)_{rs} = \begin{cases}
y_j x_i & \textrm{ if } r = \hat{\sigma}(j) \textrm{ and } s = \check{\sigma}(i) \\
0 & \textrm{ otherwise}.
\end{cases}$$
so that $P (E_{i,j})^T Q = x_i y_j E_{\hat{\sigma}(j), \check{\sigma}(i)} = T(E_{i,j})$ by Lemma~\ref{lem:D} again.
\end{proof}

Theorem~\ref{t2} motivates some further natural definitions. 

  We say that a map $T : M_n(\S) \to M_n(\S)$ \textit{exchanges} two binary relations $\tau$ and $\rho$ if $x \tau y \implies T(x) \rho T(y)$ and $x \rho y \implies T(x) \tau T(y)$. We say that it
\textit{strongly} exchanges $\tau$ with $\rho$ if $x \tau y \iff T(x) \rho T(y)$ and $x \rho y \iff T(x) \tau T(y)$. Notice that a map
$T : M_n(\S) \to M_n(\S)$ (strongly) preserves $\L$ and $\R$ if and only if the map $T' : X \mapsto T(X)^T$ (strongly) exchanges
$\L$ with~$\R$.

\begin{corollary}
Let $\S$ be an anti-negative semifield.
If a bijective $\S$-linear transformation on $M_n(\S)$ preserves any of $\D$, $\J$ and $\leq_\J$ then it:
\begin{itemize}
\item strongly preserves $\D$, $\J$, $\leq_\J$ and $\H$; and
\item \textbf{either} strongly preserves $\L$, $\R$, $\leq_L$ and $\leq_\R$ \textbf{or else} strongly exchanges $\L$ with $\R$ and strongly exchanges $\leq_{\L}$ with $\leq_{\R}$.
\end{itemize}
\end{corollary}

\begin{corollary}
Let $\S$ be an anti-negative semifield and $T:M_n(\S)\to M_n(\S)$ a bijective $\S$-linear transformation. Then the following are equivalent:
\begin{itemize}
\item[(i)]  $T$ exchanges $\L$ with $\R$;
\item[(ii)]   $T$ exchanges $\leq_{\L}$ with $\leq_{\R}$;
\item[(iii)]   $T$ strongly exchanges $ {\L}$ with $ {\R}$;
\item[(iv)]   $T$ strongly exchanges $\leq_{\L}$ with $\leq_{\R}$;
\item[(v)] there exist invertible (monomial) matrices $P,Q\in M_n(\S)$ such that for all $X\in M_n(\S)$, we have $T(X)=PX^{ T} Q$.
\end{itemize}
\end{corollary}

\begin{corollary}
Let $\S$ be an anti-negative semifield.
If a bijective $\S$-linear transformation on $M_n(\S)$ preserves any of $\L$, $\R$, $\leq_\L$ and $\leq_R$ then it strongly preserves
$\L$, $\R$, $\D$, $\J$, $\H$, $\leq_\L$, $\leq_\R$ and $\leq_J$.
\end{corollary}

\section{The $\H$ relation}
So far we have proved relatively little about the $\H$-relation. There are reasons to expect that it might behave like the $\D$, $\J$ and $\leq_\J$ relations. We have not been able to establish this at quite the same level of generality as our results above (that is, for all anti-negative semifields),
but we shall show this does happen for many anti-negative semifields including the key examples of the boolean semifield $\B$ and the tropical semifield $\mathbb{R}_{\rm_{max}}$. We begin with a lemma analogous to, but weaker than, Lemmas~\ref{lem:LR} and \ref{lem:D} above.

\begin{lemma} \label{lem:H}
Let $\S$ be an anti-negative semifield and let $T:M_n(\S)\to M_n(\S)$ be a bijective $\S$-linear transformation. If $T$  preserves the $\H$ relation then there exist $\check{\sigma}, \hat{\sigma} \in {\rm Sym}([n])$ and non-zero elements $\alpha_{i,j} \in \S$ such that either
$$T(E_{i,j}) = \alpha_{i,j} E_{\check{\sigma}(i),\hat{\sigma}(j)} \mbox { for all } i, j \in [n]; \mbox{ or }$$
$$T(E_{i,j}) = \alpha_{i,j} E_{\hat{\sigma}(j),\check{\sigma}(i)} \mbox{ for all } i,j \in [n].$$
\end{lemma}
\begin{proof}
By Lemma~\ref{lem:BG}  there exist $\alpha_{i,j} \in \S$ and $\sigma \in {\rm Sym}([n]\times [n])$ such that $T(E_{i,j}) = \alpha_{i,j}E_{\sigma(i,j)}$ for all $i,j \in [n]$. Define row sets and column sets of basic matrices as in the proof of Lemma~\ref{lem:D}. Just as there, we need to show that $T$ maps
row sets and column sets to row sets and column sets.

Suppose not. Then there must exist $i,j,k,l \in [n]$ with $i \neq j$ and $k \neq l$ such that either
\begin{itemize}
\item $\sigma(i,k)=(m,q)$ and $\sigma(j,l)=(p,q)$; or
\item $\sigma(i,k)=(q,m)$ and $\sigma(j,l)=(q,p)$,
\end{itemize}
for some $m,p,q$ with $m \neq p$. We consider the first case, the second being dual.

Consider the matrices $A =E_{i,k} + E_{j,l}$ and $B=E_{i,l} + E_{j,k}$. It is easy to see (for example by Theorem~\ref{thm:HK}) that $A\H B$. Now $T(A)$ has non-zero entries only in the $(m,q)$ and $(p,q)$ positions, and it is also easy to show that any matrix $\H$-related to $T(A)$ must have non-zero entries in these same two positions. Since $T$ preserves the $\H$ relation we have $T(A) \H T(B)$, so $T(B)$ must have this form, thus we must have $\lbrace \sigma(i,l), \sigma(j,k) \rbrace = \lbrace (m,q), (p,q) \rbrace = \lbrace \sigma(i,k), \sigma(j,l) \rbrace$ which clearly contradicts the fact that $\sigma$ is a permutation of $[n] \times [n]$.

This establishes that $T$ maps row sets and column sets to row sets and column sets. We now reason exactly as in the proof of Lemma~\ref{lem:D} to deduce that it either maps row sets to row sets and column sets to column sets (the standard case), or row sets to column sets and column sets to row sets (the transpose case). The construction of the permutations $\hat\sigma$ and $\check\sigma$ is then also just as in the proof of Lemma~\ref{lem:D}.
\end{proof}

It is instructive to compare Lemma~\ref{lem:H} with Lemmas~\ref{lem:LR} and \ref{lem:D}, which played a key role in establishing
our results for the other relations. Recall that given a bijective $\S$-linear map $T : M_n(\S) \to M_n(\S)$ preserving one of Green's relations, Lemma~\ref{lem:BG} allows us to describe $T$ by a permutation $\sigma$ of $[n] \times [n]$ and invertible constants $\alpha_{i,j}$ such that $T(E_{i,j}) = \alpha_{i,j}E_{\sigma(i,j)}$ for all $i$ and $j$.
Lemma~\ref{lem:LR} (where $T$ preserves $\L$, $\R$, $\leq_{\L}$ or $\leq_{\R}$) and Lemma~\ref{lem:D} (where $T$ preserves $\D$, $\J$ or $\leq_{\J}$) then allow us to split the permutation $\sigma$ of $[n] \times [n]$ into two permutations of $[n]$, and split the coefficients $\alpha_{i,j}$ as $x_i y_j$ for non-zero elements $x_i$ and $y_j$. Lemma~\ref{lem:H} (where $T$ preserves $\H$) performs the first of these tasks but not the second. We
do not know if it is possible to split the coefficients where $T$ preserves $\H$ over a completely general anti-negative semifield $\S$. The essence of the proof of the following
lemma is that we can do so, and hence obtain a characterisation of $\H$-preserving linear bijections, unless $\S$ admits a $2 \times 2$ matrix with an
extremely strong combination of properties. In fact we do not know whether a matrix of this kind can exist over an anti-negative semifield; we shall see that it certainly cannot exist over well-studied examples such as the tropical and boolean semifields.

\begin{lemma}\label{lem:sticky}
Let $\S$ be an anti-negative semifield and let $T : M_n(\S) \to M_n(\S)$ be a bijective $\S$-linear transformation. If $T$ preserves $\H$ then
either
\begin{itemize}
\item[(A)] there exist invertible (monomial) matrices $P,Q\in M_n(\S)$ such that either $T(X)=PXQ$ for all $X \in M_n(\S)$ or $T(X)=P X^T Q$ for all $X \in M_n(\S)$; or
\item[(B)] there exists $M = \left(\begin{array}{c c} a & b \\c& d\end{array}\right) \in M_2(\S)$ such that:
\begin{itemize}
\item[(S1)] $a,b,c,d$ are invertible in $\S$;
\item[(S2)] $M$ has factor rank $2$; and
\item[(S3)] for every invertible element $k \in \S$, the matrices
$$A_k = \left(\begin{array}{c c} ak & b \\c& dk\end{array}\right) \textrm{ and }  B_k = \left(\begin{array}{c c} a & bk \\ck& d\end{array}\right)$$
are $\mathcal{H}$-related in $M_2(\S)$.
\end{itemize}
\end{itemize}
\end{lemma}

\begin{proof}
Suppose condition (B) is not satisfied; we shall show that condition (A) must be.

Since $T$ preserves $\H$, Lemma \ref{lem:H} tells us that there exist permutations $\check{\sigma}$ and $\hat{\sigma}$ of $\{1, \ldots, n\}$ and non-zero elements $\alpha_{l,p} \in \S$ such that either:
\begin{itemize}
\item $T(E_{l,p}) = \alpha_{l,p}E_{\check{\sigma}(l), \hat{\sigma}(p)}$, for all $l, p \in [n]$ (the standard case); or
\item $T(E_{l,p}) = \alpha_{l,p}E_{\hat{\sigma}(p), \check{\sigma}(l), }$, for all $l, p \in [n]$ (the transpose case).
\end{itemize}
Following the same lines of reasoning given in the proof of Lemma \ref{lem:D} and Theorem \ref{t2}, if we let $R$ be the matrix with $R_{l,p} = \alpha_{l,p}$, then it suffices to show that every row of $R$ can be written as a multiple of some common row vector.

Suppose first that we are in the standard case. Let $l,m,p,q \in [n]$ and define
$$M = \left(\begin{array}{c c} \alpha_{l,p} & \alpha_{l,q} \\ \alpha_{m,p} & \alpha_{m,q} \end{array}\right).$$
Notice that $M$ satisfies condition (S1).
Now for any invertible $k \in \S$ consider the matrices:
\begin{eqnarray*}
U &=& kE_{l,p} + E_{l,q} + E_{m,p} +kE_{m,q}\\
V &=& E_{l,p} + kE_{l,q} + kE_{m,p} +E_{m,q}.
\end{eqnarray*}
Up to reordering, $U$ and $V$ have the same rows, and also the same columns. Thus $U \H V$,
and so by assumption $T(U) \H T(V)$. Now we have
\begin{eqnarray*}
T(U) &=& \alpha_{l,p}kE_{i,r} + \alpha_{l,q}E_{i,s} + \alpha_{m,p}E_{j,r} +\alpha_{m,q}kE_{j,s}\\
T(V) &=& \alpha_{l,p}E_{i,r} + \alpha_{l,q}kE_{i,s} + \alpha_{m,p}kE_{j,r} +\alpha_{m,q}E_{j,s},
\end{eqnarray*}
where $\check{\sigma}(l)=i$, $\check{\sigma}(m)=j$, $\hat{\sigma}(p)=r$ and $\hat{\sigma}(q)=s$.
Notice that the $2 \times 2$ submatrices of $T(U)$ and $T(V)$ obtained by restricting to the non-zero rows and columns (that is, rows $i$ and $j$ and columns $r$ and $s$) are exactly the matrices $A_k$ and $B_k$ from the statement of the Lemma. It follows easily that $A_k \H B_k$.

Since $k$ was a general non-invertible element this means that $M$ satisfies condition (S3). Since by assumption condition (B) does not hold, it must be that $M$ fails to satisfy condition (S2). In other words,  $M$ must have rank strictly less than $2$, which by Remark~\ref{rem_rank2} means that 
$\alpha_{l,p} \alpha_{m,q} = \alpha_{l,q} \alpha_{m,p}$. Fixing $l,m,p$ and letting $q$ vary, we deduce that row $l$ of $R$ can be obtained from row $m$ of $R$ via multiplication by $\alpha_{l,p}(\alpha_{m,p})^{-1}$. Since $l$ and $m$ were arbitrary, this shows that all the rows of $R$ are scalar multiples of each other, as required.

In the transpose case we use an almost identical argument but working with the transpose of the above matrix $M$.
\end{proof}

\begin{lemma}
\label{rank2}
Let $\S$ be an anti-negative semifield and suppose $M \in M_2(\S)$ satisfies conditions (S1), (S2) and (S3) from the statement
of Lemma~\ref{lem:sticky}. Let $k \in \S$ be an invertible element. Then the matrices
$$A_k = \left(\begin{array}{c c} ak & b \\c& dk\end{array}\right) \textrm{ and } B_k = \left(\begin{array}{c c} a & bk \\ck& d\end{array}\right) \in M_2(\S)$$
have factor rank $2$.
\end{lemma}

\begin{proof}
By condition (S3) we have $A_k \H B_k$, and so (by Remark~\ref{rem_factorgreen}) the matrices $A_k$ and $B_k$ have the same factor rank.

Suppose for a contradiction that this common rank is not $2$.  By Remark~\ref{rem_rank2} it follows that $adk^2=bc$ and $bck^2=ad$. Since $k$ is invertible and (by condition (S1)) so too is each of $a,b,c,d$, we may deduce that $k^4=1$. It is then easily verified that the (factor rank $1$) matrix
$$k^3 A_k \ = \ \left(\begin{array}{c c} a & bk^3 \\ck^3& d\end{array}\right)$$
is $\mathcal{H}$-related to $B_k$. Since $B_k$ and $k^3 A_k$ are of factor rank $1$ with all entries non-zero and equal diagonals, the only way this can happen is if $B_k= k^3 A_k$. Looking at the entries of the two matrices, we see that this forces $k^2=1$ and hence $ad=bc$. By Remark~\ref{rem_rank2} again this means that $M$ has factor rank 1, which contradicts condition (S2).
\end{proof}

\begin{theorem}
\label{thm:squareroot}
Let $\S$ be an anti-negative semifield in which every invertible element has a square root (for example, the boolean semifield
or tropical semifield). Let $T : M_n(\S) \to M_n(\S)$ be
a bijective $\S$-linear transformation. Then $T$ preserves $\H$ if and only if there exist invertible (monomial) matrices $P,Q\in M_n(\S)$ such that either $T(X)=PXQ$ for all $X \in M_n(\S)$ or $T(X)=P X^T Q$ for all $X \in M_n(\S)$. In particular, $T$ preserves $\H$ if and only if $T$ preserves
$\D$.
\end{theorem}

\begin{proof}
It is easy to see that maps of the given form preserve $\H$. For the converse, suppose $T$ is an $\S$-linear bijection which preserves $\H$, and
which does not have the given form. Then by Lemma~\ref{lem:sticky} there is a matrix
$$M = \left(\begin{array}{cc} a & b \\ c & d \end{array}\right) \in M_2(\S)$$
satisfying the conditions (S1), (S2) and (S3) given in
the Lemma. By condition (S1), the entries are all invertible, so $bc a^{-1} d^{-1}$ exists and is invertible.
Let $k$ be the square root of $bca^{-1}d^{-1}$. Clearly $k$ is invertible (with inverse $kdac^{-1}b^{-1}$). Let $A_k$ be as defined
in Lemma~\ref{lem:sticky}. Now $bca^{-1}d^{-1} = k^2$ implies $(ak)(dk) = bc$ which by Remark~\ref{rem_rank2} means that
$A_k$ has factor rank $1$. But Lemma~\ref{rank2} says that $A_k$ has factor rank $2$, giving a contradiction.
\end{proof}

The hypotheses of Theorem~\ref{thm:squareroot} apply in the usual tropical semifield over the real numbers but not, for example, the tropical semifield restricted to the integers. The following is a comparable result encompassing this and other similar cases.

\begin{theorem}
Let $\S$ be an anti-negative semifield. Suppose that whenever $\lbrace u, v \rbrace \subseteq \S^2$ and $\lbrace x, y \rbrace \subseteq \S^2$ are different generating sets for the same $2$-generated submodule of $\S^2$, we have $\lbrace u, v \rbrace = \lbrace \lambda x, \mu y \rbrace$
for some invertible $\lambda, \mu \in \S$.

Then an $\S$-linear bijection $T : M_n(\S) \to M_n(\S)$ preserves $\H$ if and only if there exist invertible (monomial) matrices $P,Q\in M_n(\S)$ such that either $T(X)=PXQ$ for all $X \in M_n(\S)$ or $T(X)=P X^T Q$ for all $X \in M_n(\S)$. In particular, $T$ preserves $\H$ if and only if $T$ preserves
$\D$.
\end{theorem}

\begin{proof}
If $\S = \mathbb{B}$ then the result follows from Theorem~\ref{thm:squareroot}, so assume $S \neq \mathbb{B}$.
Just as in the proof of Theorem~\ref{thm:squareroot}, it is clear that maps of the given form preserve $\H$, so suppose $T$ is an $\S$-linear bijection preserving $\H$ which doesn't have the given form. As before, by Lemma~\ref{lem:sticky} there is a matrix
$$M = \left(\begin{array}{cc} a & b \\ c & d \end{array}\right) \in M_2(\S)$$
satisfying the conditions (S1), (S2) and (S3) given in
the Lemma. For each invertible $k \in S$ let $A_k$ and $B_k$ be as given in Lemma~\ref{lem:sticky}. By Lemma \ref{rank2} all the matrices
$A_k$ and $B_k$ have factor rank $2$, from which it follows that the row space of each such matrix is a $2$-generated submodule of $\S^2$.  Consider the row space of $A_k$, which (since $A_k \H B_k$) is equal to the row space of $B_k$.

Applying the assumption on $\S$ where $\{u,v\}$ is the set of rows of $A_k$ and $\{x,y\}$ is the set of rows of $B_k$, it now follows that there exist invertible $\lambda,\mu \in \S$ such that either
\begin{eqnarray*}
(a, bk) &=& \lambda (c, dk) \mbox{ and } (ck, d) = \mu (ak, b), \; \; \mbox{ or }\\
(a, bk) &=& \lambda (ak, b) \mbox{ and } (ck, d) = \mu (c, dk).
\end{eqnarray*}
Since $a,b,c,d,k,\lambda,\mu$ are all invertible, in the first case we deduce that $\lambda=ac^{-1}=bd^{-1}$, so that $ad = bc$ which by
Remark~\ref{rem_rank2} contradicts condition (S2).
In the second case we deduce that $\lambda=\mu=k$ and $k^2=1_{\S}$. But since $k$ was a general invertible element and $\S \neq \mathbb{B}$,
this contradicts Proposition~\ref{prop_notallroots}.
\end{proof}

\section*{Acknowledgments}

The research contained in this article was started in 2012 during a visit of the first author to the University of Manchester, funded by EPSRC grant EP/I005293/1 (\textit{Nonlinear Eigenvalue Problems: Theory and Numerics}). He is grateful to the School of Mathematics and the Tropical Mathematics Group for their warm hospitality.  He also thanks RFBR grant 15-01-01132 for partial financial support of his research.

\end{document}